\documentclass[12pt]{article}
\usepackage{authblk}
\usepackage{amsmath}
\usepackage{amssymb}
\usepackage{amsthm}
\usepackage{algorithm}
\usepackage{algorithmicx}
\usepackage{algpseudocode}
\usepackage[]{geometry}
\newtheorem{assumption}{Assumption}[section]
\newtheorem{proposition}{Proposition}[section]
\newtheorem{lemma}{Lemma}[section]
\newtheorem{thm}{Theorem}[section]
\newtheorem{remark}{Remark}[section]

\title{Optimal distributions for randomized unbiased estimators with an infinite horizon and an adaptive algorithm}
\author{Chao Zheng}
\author{Jiangtao Pan}
\author{Qun Wang}
\affil[]{School of Data Sciences, Zhejiang University of Finance and Economics, Hangzhou, China \thanks{Email: chao.zheng12@gmail.com; 211510011019@zufe.edu.cn; wangqun@zufe.edu.cn.}}
\date{}

\begin{document}
\maketitle

\begin{abstract}
The randomized unbiased estimators of Rhee and Glynn (Operations Research:
63(5), 1026-1043, 2015) can be highly efficient at approximating expectations of path functionals associated with stochastic differential equations (SDEs). However, there is a lack of algorithms for calculating the optimal distributions with an infinite horizon. In this article, based on the method of Cui et.al. (Operations Research Letters: 477-484, 2021), we prove that, under mild assumptions, there is a simple representation of the optimal distributions. Then, we develop an adaptive algorithm to compute the optimal distributions with an infinite horizon, which requires only a small amount of computational time in prior estimation. Finally, we provide numerical results to illustrate the efficiency of our adaptive algorithm.


\end{abstract}

\section{Introduction}
Monte Carlo methods are useful for approximating expectations of functionals of stochastic processes, if there are no analytical solutions. For a standard Monte Carlo method, when the functional of the underlying stochastic processes is sampled exactly, the convergence rate of the mean squared error (MSE) is $O(c^{-1})$, where $c$ is the computational cost. In this article, we consider underlying stochastic processes following certain stochastic differential equations (SDEs), which are difficult to generate exactly. In this case, we may resort to a time-discrete scheme (e.g., the Euler scheme, the Milstein scheme) to obtain approximate values. The classical time-discrete schemes can be found in Kloeden and Platen \cite{KP}. However, those time-discrete schemes, associated with a standard Monte Carlo method, usually lead to biased Monte Carlo estimators. For such a biased estimator, the convergence rate is lower than $O(c^{-1})$, although by carefully selecting the step size of the time-discrete scheme and the number of samples as functions of the computational cost $c$, the bias and the variance of the Monte Carlo estimator can be balanced (see Duffie and Glynn \cite{DG}). 

Rhee and Glynn \cite{RG} made a breakthrough by constructing several unbiased Monte Carlo estimators when the underlying stochastic processes are approximated using time-discrete schemes. A similar idea was considered in McLeish \cite{M}). These estimators recover the convergence rate $O(c^{-1})$ and can be regarded as unbiased versions of the multilevel Monte Carlo estimators of Giles (see Giles \cite{Gi}). They can easily combine with any time-discrete scheme that is convergent in $L^2$ norm with a sufficiently high order and hence have many applications. Since then, the idea of unbiased estimators has been extended to more complicated settings. Glynn and Rhee \cite{GR} considered the application of unbiased estimators for Markov chain equilibrium expectations. Blanchet and Glynn \cite{BG} investigated the exact simulation for functions of expectations. Vihola \cite{V} presented a more general class of unbiased estimators. Zheng and Glynn \cite{ZG} developed a central limit theorem for infinitely stratified unbiased estimators. 

The construction of unbiased estimators in Rhee and Glynn \cite{RG} involves determining the optimal distribution of a random variable $N$. This is an infinite-horizon optimization problem subject to certain constraints. Rhee and Glynn \cite{RG} proposed an $m$-truncated dynamic programming algorithm to find the optimal distribution of $N$ in $O(m^3)$ operations. Cui et.al. \cite{CLZZ} improved this result by providing an algorithm with order $O(m)$. However, the above algorithms are valid for merely the $m$-truncated optimal distribution instead of the optimal distribution with an infinite horizon. In practical applications, one may heuristically choose a value of $m$, run an algorithm for the $m$-truncated optimal distribution and finally heuristically select the tail distribution (or simply ignore it). Clearly, this approach may not be optimal in the infinite sense. On the other hand, both the algorithm in Rhee and Glynn \cite{RG} and that in Cui et.al. \cite{CLZZ} rely on a prior estimation of $m$ sample variances. Although choosing a large value of $m$ would typically make the $m$-truncated optimal distribution close to the optimal distribution with an infinite horizon, it can be computationally expensive in prior estimation when $m$ is large. 

In this article, we provide a solution to address these two problems. Based on the fact that many time-discrete schemes converge in $L^2$ norm with a sufficiently high order, we propose mild assumptions on the convergent behaviour of these schemes and prove that there is a simple representation of the optimal distribution, and, in particular the optimal tail distribution. The proof is based on the optimization method in Cui et.al. \cite{CLZZ} and a careful analysis of some structures of the optimal distribution. For practical applications, we develop an adaptive algorithm, which is an extension of the algorithm in Cui et.al. \cite{CLZZ} obtained by adding an adaptive value of $m$ and the optimal tail distribution. In our numerical experiment, we find that a small value of $m$ typically suffices to produce a highly accurate approximation of the optimal distribution with an infinite horizon, saving a large amount of computation time for the prior estimation. 

The remainder of the article is organized as follows: In Section 2, we review the unbiased estimators in Rhee and Glynn \cite{RG}. Section 3 reviews the method of Cui et.al. \cite{CLZZ}. In Section 4, we derive the optimal distribution with an infinite horizon and propose an adaptive algorithm. Section 5 reports numerical results to illustrate the efficiency of our algorithm. Finally, we conclude the article in Section 6.

\section{Randomized unbiased estimators}

Let $(X(t): t\geq 0)$ be the unique solution to the following SDE
\[
dX(t)=\mu(X(t))dt+\sigma(X(t))dB(t),
\]
where $\mu: \mathbb{R}^d\rightarrow \mathbb{R}^d$, $\sigma: \mathbb{R}^d\rightarrow \mathbb{R}^{d\times m}$ and $(B(t): t\geq 0)$ is an $m$-dimensional standard Brownian motion. In many applications, one needs to calculate the expectation $E[f(X)]$, where $f$ is a functional of $X$. 

In general, it is difficult to generate $X$ and $f(X)$ exactly. Hence, one may use a time-discrete scheme for approximation. The simplest approximation is the Euler scheme
\[
X_h((j+1)h)=X_h(jh)+\mu(X_h(jh))+\sigma(X_h(jh))(B((j+1)h)-B(jh)), 
\]
where $h$ is the step size and $X_{h}(0)=X(0)$. Here, $X_h$ and $f(X_h)$ are approximations of $X$ and $f(X)$ respectively. However, a typical time-discrete scheme is biased, i.e., $E(f(X_{h}))\neq E(f(X))$, although $E(f(X_{h}))$ may converge to $E(f(X))$ as $h$ goes to $0$. In this case, the convergence rate of the MSE of a standard Monte Carlo estimator associated with a time-discrete scheme is lower than $O(c^{-1})$ (see Duffie and Glynn \cite{DG}), where $c$ is the computational cost. This section reviews the unbiased estimators introduced by Rhee and Glynn \cite{RG}, which can achieve the canonical convergence rate $O(c^{-1})$ in the above setting. 

Let $L^2$ be the Hilbert space of square integrable random variables, and let $Y\in L^2$ (i.e., $E(Y^2)<\infty$). It may be difficult to generate $Y$ in finite time, but we assume that there is a sequence of approximations $(Y_n:n\geq0)$, that can be generated in finite time and satisfy $\lim\limits_{x\to+\infty}E[(Y_n-Y)^{2}]=0$. Let $N$ be a nonnegative integer-valued random variable, that is independent of $Y_n$, $n\geq0$, and let
\[
Z=\sum_{n=0}^{N}\frac{\Delta_n}{P(N\geq n)}
\]
where $\Delta_n=Y_n-Y_{n-1}$ with $Y_{-1}=0$ and $P$ is the probability of $N$. We call $Z$ the coupled sum estimator. The condition for $Z$ to be an unbiased estimator is provided by Theorem \ref{thm2.1} from Theorem 1 in Rhee and Glynn \cite{RG}:

\begin{thm} \label{thm2.1}
If
\begin{equation}
\sum_{n=1}^{\infty}\frac{E[(Y_n-Y)^{2}]}{P(N\geq n)}<\infty
\label{9}
\end{equation}
then $Z\in L^2$ is an unbiased estimator of $E(Y)$, and
\[
E(Z^2)=\sum_{n=0}^{\infty}\frac{v_n}{P(N\geq n)}
\]
where $v_n=E[(Y_{n-1}-Y)^2]-E[(Y_n-Y)^2]$.
\end{thm}

Rhee and Glynn \cite{RG} introduced another unbiased estimator $\widetilde{Z}$ of $E(Y)$ defined as
\[
\widetilde{Z}=\sum_{n=0}^{N}\frac{\widetilde{\Delta}_n}{P(N\geq n)},
\]
where $\widetilde{\Delta}_n=\widetilde{Y}_n-\widetilde{Y}_{n-1}$. Here, $(\widetilde{Y}_n,\widetilde{Y}_{n-1})$ has the same marginal distribution as $(Y_n,Y_{n-1})$, but $\widetilde{\Delta}_n$ for each $n$ is independent. This estimator is called the independent sum estimator. The following theorem, from Theorem 2 in Rhee and Glynn \cite{RG}, guarantees that $\widetilde{Z}$ is unbiased.

\begin{thm}
If (\ref{9}) holds, then $\widetilde{Z}\in L^2$ is an unbiased estimator of $E(Y)$. Furthermore, 
\[
E(\widetilde{Z}^2)=\sum_{n=0}^{\infty}\frac{\widetilde{v}_n}{P(N\geq n)}
\]
where $\widetilde{v}_n=var(Y_{n}-Y_{n-1})+(E(Y)-E(Y_{n-1}))^2-(E(Y)-E(Y_{n}))^2$.
\end{thm}

There is a third unbiased estimator introduced, referred to as the single term estimator; see Rhee and Glynn \cite{RG} for more discussions. In this article, we focus on the coupled sum estimator. The analysis of the independent sum estimator is very similar.

Note that the distribution for $N$ needs to be specified. The minimum requirement is to choose a distribution that satisfies (\ref{9}). To maximize the efficiency of the unbiased estimator $Z$, Rhee and Glynn \cite{RG} proposed finding a distribution that minimizes the product $E(\tau) \times var(Z)$. Here, $\tau$ represents the computational time required to generate a sample of $Z$ and
\[
E(\tau)=E\left(\sum_{n=0}^{N}t_n\right)=\sum_{n=0}^{\infty}t_{n}P(N\geq n),
\]
where $t_{n}$ is the expected computational time to calculate $Y_{n}$; For the variance $var(Z)$, it is obtained from Theorem \ref{thm2.1} that
\[
var(Z)=E(Z^2)-(E(Z))^2=\sum_{n=0}^{\infty}\frac{v_n}{P(N\geq n)}-(E(Y))^{2}.
\]
Therefore, the problem of minimizing $E(\tau) \times var(Z)$ can be written as the following optimization problem:
\begin{equation}
\begin{aligned} \label{Problem1}
\min\limits_{F}\quad& g(F):=\left(\sum_{n=0}^{\infty}\frac{\beta_n}{F_n}\right)\left(\sum_{n=0}^{\infty}t_n F_n\right)\\
s.t.\quad & F_i\geq F_{i+1},\forall i\geq0\\
& F_i>0,\forall i\geq0\\
& F_0=1,
\end{aligned}
\end{equation}
where $\beta_{0}=v_{0}-(E(Y))^2$, $\beta_{n}=v_{n}$, $n\geq 1$ and $F_{n}=P(N\geq n)$. We assume that $(\beta_i,i\geq0)$ is a nonnegative sequence and $t_i,i\geq0$ are bounded below by a positive constant, so that there exists a solution to the optimization problem; see Proposition 2 in Rhee and Glynn \cite{RG}.

Since in general, there is no simple formula to calculate the optimal $F_{n}$ for all $n\geq 0$, Rhee and Glynn \cite{RG} considered a finite version of the optimization problem as follows:
\begin{equation}
\begin{aligned}\label{Problem2}
\min\limits_{F}\quad& g(F):=\left(\sum_{n=0}^{m}\frac{\beta_n}{F_n}\right)\left(\sum_{n=0}^{m}t_n F_n\right)\\
s.t.\quad & F_i\geq F_{i+1},\forall i\geq0\\
& F_i>0,\forall i\geq0\\
& F_0=1
\end{aligned}
\end{equation}
for any $m\in \mathbb{N}$, and then provided a dynamic programming algorithm to compute the optimal $F$ with the order $O(m^3)$. We call the solution to the optimization problem (\ref{Problem2}) the ``$m$-truncated" optimal distribution.  

However, it is unclear whether the unbiased estimator $Z$ could potentially lose efficiency if one uses the ``$m$-truncated" optimal distribution instead of the optimal distribution with an infinite horizon. On the other hand, although heuristically one could make the ``$m$-truncated" optimal distribution sufficiently close to the optimal distribution by increasing the value of $m$, the calculation of the ``$m$-truncated" optimal distribution requires a prior estimation of $\beta_{n}$, $n=0,1,...,m$, which can be time-consuming when $m$ is large. Furthermore, with the same number of Monte Carlo samples, estimating $\beta_{n}$ with a high $n$ takes much more time than estimating it with a low $n$. This imposes a challenge: how to develop an algorithm to handle the optimal distribution with a low computation cost in prior estimation.

\section{M-truncated optimal distribution}

This section reviews the method proposed in Cui et.al. \cite{CLZZ} to calculate the ``$m$-truncated" optimal distribution in the optimization problem (\ref{Problem2}). Our analysis of the optimal distribution with an infinite horizon and the development of an adaptive algorithm are based on their results.

The main approach in Cui et.al. \cite{CLZZ} is to convert the original optimization problem to its dual problem, which is formulated as
\begin{equation}
\begin{aligned}\label{Dualpro}
\min\limits_{F}\quad& \sum_{n=0}^{m}\left(\frac{\beta_n}{F_n}+\mu t_n F_n\right)\\
s.t.\quad & F_i\geq F_{i+1},\forall i\geq0\\
& F_i>0,\forall i\geq0\\
\end{aligned},
\end{equation}
with $\mu>0$. Let $\widetilde{F}=(\widetilde{F}_{0},\widetilde{F}_{1},...,\widetilde{F}_{m})$ be the solution. The connection between the original optimization problem and its dual problem is represented as Proposition \ref{Dual} below:

\begin{proposition} \label{Dual}
If $\beta_{n}$ and $t_{n}$ are positive for all $n$, then the solution $\widetilde{F}$ to the dual problem (\ref{Dualpro}) exists and for any $\mu>0$, 
\[
F^*=\frac{\widetilde{F}}{\widetilde{F}_{0}}
\]
is the solution to the optimization problem (\ref{Problem2}).
\end{proposition}
\begin{proof}
See Theorems 2.2 and 2.3 in Cui et.al. \cite{CLZZ} and the related proofs in their supplementary material.
\end{proof}

Proposition \ref{Dual} means that computing the $m$-truncated optimal distribution in Problem (\ref{Problem2}) is effectively the problem of computing $\widetilde{F}$ in Problem (\ref{Dualpro}). Cui et.al. \cite{CLZZ} provided an algorithm to calculate $\widetilde{F}$. Their algorithm has the computational complexity $O(m)$, which improves on the dynamic programming algorithm in Rhee and Glynn \cite{RG} with the computational complexity $O(m^3)$. 

Throughout this article, we always assume that the assumption in Theorem \ref{Dual} is satisfied (i.e., $\beta_{n}$ and $t_{n}$ are positive for all $n$). Let
\[
f_{n}(F_{n})=\frac{\beta_n}{F_n}+\mu t_n F_n.
\]
It is easy to show that $f_{n}$ is convex, and the minimum of $f_{n}$ can be achieved at $x_{n}:=\sqrt{\frac{\beta_{n}}{\mu t_{n}}}$ by solving the equation $f^{'}(x_{n})=0$. Let us move on to a slightly more complicated problem with two variables:
\[
\begin{aligned}
\min\limits_{F_{n}, F_{n+1}}\quad& \left[f_{n}(F_{n})+f_{n+1}(F_{n+1})\right]\\
s.t.\quad & F_n\geq F_{n+1}>0,\\
\end{aligned}
\]
and let its solution be $(\widetilde{F}_{n},\widetilde{F}_{n+1})$.  Let $(x_n,x_{n+1})=\left(\sqrt{\frac{\beta_{n}}{\mu t_{n}}},\sqrt{\frac{\beta_{n+1}}{\mu t_{n+1}}}\right)$. If $x_{n}> x_{n+1}$, then $(\widetilde{F}_{n},\widetilde{F}_{n+1})=(x_n,x_{n+1})$; otherwise, we obtain $(\widetilde{F}_{n},\widetilde{F}_{n+1})=(x^*,x^*)$, where $x^*$ is the solution to the following problem
\[
\min\limits_{x}[f_n(x)+f_{n+1}(x)],
\]
and $x_n\leq x^* \leq x_{n+1}$. Precisely, $x^*$ satisfies the equation $f_{n}^{'}(x^*)+f_{n+1}^{'}(x^*)=0$, and hence
\[
x^*=\sqrt{\left(\beta_n+\beta_{n+1}\right)/\left( \mu t_n+\mu t_{n+1}\right)}.
\]
The steps above to calculate the solution $(\widetilde{F}_{n},\widetilde{F}_{n+1})$ imply that when $x_n\leq x_{n+1}$, we `combine' the two functions $f_{n}$ and $f_{n+1}$ into a single function that is easy to analyse (i.e. we let $f_{n,n+1}:=f_{n}+f_{n+1}$ and focus on $\min\limits_{x} f_{n,n+1}(x)$). Proposition \ref{cui} provides a theoretical basis for the steps above and their generalizations, which is adapted from Proposition 2.1 in Cui et.al. \cite{CLZZ}.   
\begin{proposition} \label{cui}
Let $f_{n}(F_{n})=\frac{\beta_n}{F_n}+\mu t_n F_n$. Let $(\widetilde{F}_{s},..,\widetilde{F}_{t})$ be the solution to $\min\limits_{F_{s},.., F_{t}}\sum_{n=s}^{t}f_{n}(F_{n})$ such that $F_{s}\geq...\geq F_{t}>0$ and let $(\widetilde{F}_{t+1},..,\widetilde{F}_{l})$ be the solution to $\min\limits_{F_{t+1},.., F_{l}}\sum_{n=t+1}^{l}f_{n}(F_{n})$ with a similar constraint, where $s<l$. Suppose that $\widetilde{F}_{s}=...=\widetilde{F}_{t}=V_{1}$ and $\widetilde{F}_{t+1}=...=\widetilde{F}_{l}=V_{2}$, where
\[
V_{1}=\sqrt{\frac{\sum_{n=s}^{t}\beta_{n}}{\mu\sum_{n=s}^{t} t_{n}}}, \quad V_{2}=\sqrt{\frac{\sum_{n=t+1}^{l}\beta_{n}}{\mu\sum_{n=t+1}^{l} t_{n}}}.
\]
If $V_{1}\leq V_{2}$, then we have
\[
\bar{V}=\sqrt{\frac{\sum_{n=s}^{l}\beta_{n}}{\mu\sum_{n=l}^{l} t_{n}}}\in [V_{1},V_{2}],
\]
such that $(F_{s}^{*},..,F_{l}^{*})=(\bar{V},...,\bar{V})$ is the solution to $\min\limits_{F_{s},.., F_{l}}\sum_{n=s}^{l}f_{n}(F_{n})$ with the constraint $F_{s}\geq...\geq F_{l}>0$.
\end{proposition}

Based on Proposition \ref{cui}, Cui et.al. \cite{CLZZ} proposed an algorithm to calculate the solution to Problem (\ref{Dualpro}), and then compute the $m$-truncated optimal distribution in Problem (\ref{Problem2}) using Proposition \ref{Dual}.

\section{Optimal distribution with an infinite horizon}
Both the algorithm in Rhee and Glynn \cite{RG} and that in Cui et.al.\cite{CLZZ} can only produce the $m$-truncated optimal distribution. In this section, we focus on deriving a simple representation for the optimal distribution with an infinite horizon, if $\beta_{n}$ and $t_{n}$ satisfy certain mild assumptions.

In the context of unbiased estimators for SDE models, $Y_{n}$ is an approximation of $Y$ using a time-discrete scheme typically with step size $T/2^n$, where $T>0$ is the time horizon. Since $t_{n}$ is the expectation of the computational time needed to generate $Y_{n}$, it is clear that $t_{n}$ is proportional to $2^n$. Thus, it is reasonable to assume that $t_{n}=2^n$, $n=0,1,...$. On the other hand, for a time-discrete scheme that is strongly convergent with order $p>1/2$, i.e.,
\[
E[(Y_{n}-Y)^2]=O(2^{-2pn}),
\]
it is easy to combine it with an unbiased estimator, that has a finite variance and computational time. Then, by the definition of $\beta_{n}$, we have $\beta_{n}=O(2^{-2pn})$, which usually leads to $\beta_{n}/\beta_{n+1}\approx 4^p$ for sufficiently large $n$. Now, we are in a position to impose Assumption \ref{aspt1} below: 

\begin{assumption}\label{aspt1}
Suppose that $t_{n}=2^n$ and $\beta_{n}>0$ for all n. Suppose that there exists $m\geq1$, such that $4^p-\epsilon<\beta_{n}/\beta_{n+1}<4^p+\epsilon$ for all $n\geq m$, where $p>1/2$ and $\epsilon \in (0,1)$.
\end{assumption} 

Assumption \ref{aspt1} implies that $\beta_{n}>\beta_{n+1}>\beta_{n+2}>...$ for all $n\geq m$. In this section, we always assume that Assumption \ref{aspt1} is satisfied. Recall that the infinite-horizon optimization problem(i.e., Problem (\ref{Problem1})) is
\begin{equation}
\begin{aligned} \label{Problem3}
\min\limits_{F}\quad& g(F):=\left(\sum_{n=0}^{\infty}\frac{\beta_n}{F_n}\right)\left(\sum_{n=0}^{\infty}t_n F_n\right)\\
s.t.\quad & F_i\geq F_{i+1},\forall i\geq0\\
& F_i>0,\forall i\geq0\\
& F_0=1.
\end{aligned}
\end{equation}
Our goal is to derive a simple representation for the solution under Assumption \ref{aspt1}.

Let $J:=(L_i,i\geq0)$ be a strictly increasing integer-valued sequence such that $L_0=0$ and $\beta_{i}(J)/t_i(J)$ is a strictly decreasing function of $i$, where
\[
\beta_i(J):=\sum_{k=L_i}^{L_{i+1}-1}\beta_k, \quad t_i(J):=\sum_{k=L_i}^{L_{i+1}-1}t_k.
\]
Note that $J$ may not be unique and we let $\mathcal{J}$ be the set of all $J$. We consider the following sub-problem of Problem (\ref{Problem3}):
\begin{equation}
\begin{aligned} \label{pro2}
\min\limits_{\overline{F}}\quad& g_{J}(\overline{F}):=\left(\sum_{n=0}^{\infty}\frac{\beta_n (J)}{\overline{F}_{n}}\right)\left(\sum_{n=0}^{\infty}t_n (J)\overline{F}_{n}\right)\\
s.t.\quad & \overline{F}_i\geq \overline{F}_{i+1},\forall i\geq0\\
& \overline{F}_i>0,\forall i\geq0\\
& \overline{F}_0=1.
\end{aligned}
\end{equation}
Problem (\ref{pro2}) can be regarded as Problem (\ref{Problem3}) subject to an additional constraint
\[
F_{L_i}=F_{L_{i}+1}=...=F_{L_{i+1}-1},\quad\forall i\geq0 
\]
corresponding to $J$. For any $J\in \mathcal{J}$, Proposition 1 in Rhee and Glynn \cite{RG} implies that the solution to Problem (\ref{pro2}) is 
\[
\overline{F}_{i}=\sqrt{\frac{\beta_i(J)/t_i(J)}{\beta_0(J)/t_0(J)}},\quad\forall i\geq0.
\]
Furthermore, let $F^{*}:=(F_i^*,i\geq0)$ be the solution to Problem (\ref{Problem3}). Proposition 3 and Theorem 3 in Rhee and Glynn \cite{RG} show that there is an optimal $J^{*}\in \mathcal{J}$, such that the solution to Problem (\ref{Problem3}) is the solution to Problem (\ref{pro2}) associated with $J=J^*:=(L_i^*,i\geq0)$, i.e., the solution is
\begin{equation}
F_{L_{i}^*}^{*}=F_{L_{i}^*+1}^{*}=...=F_{L_{i+1}^*-1}^{*}=\overline{F}_{i}^{*}=\sqrt{\frac{\beta_i(J^*)/t_i(J^*)}{\beta_0(J^*)/t_0(J^*)}},\quad\forall i\geq0. \label{solutn}
\end{equation}
Thus, searching for the solution to Problem (\ref{Problem3}) is equivalent to finding the optimal $J^*$. 

Let $\mathcal{J}^{(L_1,L_2,...,L_l)}\subseteq \mathcal{J}$, such that the first $l+1$ elements of all $J\in \mathcal{J}^{(L_1,L_2,...,L_l)}$ are $L_0,L_1,...,L_l$, where $L_l\geq m$ and $\frac{\beta_{l-1}(J)}{t_{l-1}(J)}>\frac{\beta_{L_l}}{t_{L_l}}$. Proposition \ref{pro1} below provides a representation for $\mathcal{J}^{(L_1,L_2,...,L_l)}$.

\begin{proposition}\label{pro1}
Suppose that Assumption \ref{aspt1} is satisfied. If $\mathcal{J}^{(L_1,L_2,...,L_l)}$ is non-empty, then we obtain
\[
\mathcal{J}^{(L_1,L_2,...,L_l)}=\{(L_0,L_1,L_2,...,L_l,L_l+k_1,L_{l+1}+k_2,...), k_i \in \mathbb{N}^{+}\}.
\]
\end{proposition}
\begin{proof}
By definition, we have
\[
\mathcal{J}^{(L_1,L_2,...,L_l)}\subseteq\{(L_0,L_1,L_2,...,L_l,L_l+k_1,L_{l+1}+k_2,...), k_i \in \mathbb{N}^{+}\}.
\]
Thus, it suffices to prove that the inverse is also true. For any $i\geq m$, we have $\beta_i >\beta_{i+1}>...>\beta_{i+k}$ and $t_i<t_{i+1}<...<t_{i+k}$, where $k\in \mathbb{N}^{+}$, so it holds that
\[
\frac{\beta_i}{t_i}>\frac{\sum_{j=i}^{i+k}\beta_{j}}{\sum_{j=i}^{i+k}t_{j}}>\frac{\beta_{i+k}}{t_{i+k}}.
\]
Then, for any $J=(L_0,L_1,L_2,...,L_l,L_l+k_1,L_{l+1}+k_2,...), k_i \in \mathbb{N}^{+}$, we obtain
\[
\frac{\beta_{0}(J)}{t_{0}(J)}>...>\frac{\beta_{l-1}(J)}{t_{l-1}(J)}>\frac{\beta_{L_l}}{t_{L_l}}>\frac{\beta_{l}(J)}{t_{l}(J)}>\frac{\beta_{l+1}(J)}{t_{l+1}(J)}>...
\]
which implies $J\in \mathcal{J}^{(L_1,L_2,...,L_l)}$. The proof is complete.
\end{proof}

Among all $J\in\mathcal{J}^{(L_1,L_2,...,L_l)}$, the most important one is
\[
J_{*}^{(L_1,L_2,...,L_l)}:=(L_0,L_1,L_2,...,L_l,L_l+1,L_l+2,...).
\]
For notational convenience, we write $J_{*}^{(L_1,L_2,...,L_l,L_l+1,L_l+2,...)}$ as $J_{*}^{L}$.

\begin{lemma}\label{lemma1}
Suppose that Assumption \ref{aspt1} is satisfied and $\mathcal{J}^{(L_{1},L_{2},...,L_{l})}$ is non-empty. Then for all $J\in \mathcal{J}^{(L_{1},L_{2},...,L_{l})}$, we have
\[
\min_{\overline{F}}  g_{J}(\overline{F})\geq \min_{\overline{F}}  g_{J_{*}^L}(\overline{F}).
\]
\end{lemma}
\begin{proof}
For $J_{*}^{L}\in \mathcal{J}^{(L_{1},L_{2},...,L_{l})}\subseteq \mathcal{J}$, substituting (\ref{solutn}) into (\ref{pro2}) yields
\[
\min_{\overline{F}}  g_{J_{*}^L}(\overline{F})=\left(\sum_{n=0}^{\infty}\sqrt{\beta_n(J_{*}^{L})t_{n}(J_{*}^{L})}\right)^2.
\]
Let $J=(L_{0},L_{1},L_{2},...,L_{l},L_{l}+2,L_{l}+3,...)$. Since 
\[
\sqrt{(a_1+a_2)(b_1+b_2)}\geq \sqrt{a_1 b_1}+\sqrt{a_2 b_2}
\]
for any $a_1,a_2,b_1,b_2\geq 0$, we can easily show that
\begin{align*}
\min_{\overline{F}}  g_{J}(\overline{F})&=\left(\sum_{n=0}^{l-1}\sqrt{\beta_n(J_{*}^{L})t_{n}(J_{*}^{L})}+\sqrt{(\beta_l(J_{*}^{L})+\beta_{l+1}(J_{*}^{L}))(t_{l}(J_{*}^{L})+t_{l+1}(J_{*}^{L}))}\right.\\
&\quad+\left.\sum_{n=l+2}^{\infty}\sqrt{\beta_n(J_{*}^{L})t_{n}(J_{*}^{L})}\right)^2\geq\min_{\overline{F}}  g_{J_{*}^L}(\overline{F}).
\end{align*}
The inequality above can be interpreted as when we combine the $l$-th and $(l+1)$-th terms of $J_{*}^{L}$, the value of $\min_{\overline{F}}  g_{J}(\overline{F})$ increases. 

More generally, for any $k_{0},k_{1},...,k_{M}\in \mathbb{N}^+$, let
\[
J_{1}=(L_{0},L_{1},...,L_{l},L_{l}+k_{0},...,L_{l+M}+k_{M},L_{l+M}+k_{M}+1,L_{l+M}+k_{M}+2,...),
\]
and
\[
J_{2}=(L_{0},L_{1},...,L_{l},L_{l}+k_{0},...,L_{l+M}+k_{M},L_{l+M}+k_{M}+2,L_{l+M}+k_{M}+3,...),
\]
where $J_{2}$ can be regarded as combining the $(l+M)$-th and $(l+M+1)$-th terms of $J_{1}$. We can show analogously that 
\[
\min_{\overline{F}}  g_{J_{2}}(\overline{F})\geq\min_{\overline{F}}  g_{J_{1}}(\overline{F}).
\]
Note that an arbitrary $J\in \mathcal{J}^{(L_{1},L_{2},...,L_{l})}$ can be obtained from $J_{*}^{L}$, if we combine the relevant terms analogously many times (up to countably infinite times) and each time we combine them, the value of $\min_{\overline{F}}  g_{J}(\overline{F})$ increases. Therefore, we can conclude that $\min_{\overline{F}}  g_{J}(\overline{F})\geq \min_{\overline{F}}  g_{J_{*}^L}(\overline{F})$ for all $J\in \mathcal{J}^{(L_{1},L_{2},...,L_{l})}$ and the proof is complete.
\end{proof}

\begin{remark}
Lemma \ref{lemma1}, together with equation (\ref{solutn}), indicates that the tail of the optimal distribution $F_{n}^*$ is proportional to $\sqrt{\frac{\beta_{n}}{t_{n}}}$ for sufficiently large $n$. Since $\beta_{n}=O(2^{-2pn})$ and $t_{n}=O(2^{n})$, we have the optimal distribution $F_{n}^*=O(2^{-(2p+1)n/2})$. In particular, when $p=1$, we obtain $F_{n}^*=O(2^{-3n/2})$.
\end{remark}

The next question is how to specify the optimal distribution of $F_{n}$ given the optimal `$m$-truncated' distribution of $F_{n}$. The answer can be found at Theorem \ref{thm4.1}.

\begin{thm} \label{thm4.1}
Let $F^*$ be the solution of Problem (\ref{Problem3}) with an infinite horizon. Let $F^*(m)$ be the solution of
\begin{equation}
\begin{aligned} \label{optm}
\min\limits_{F}\quad& g(F):=\left(\sum_{n=0}^{m}\frac{\beta_n}{F_n}\right)\left(\sum_{n=0}^{m}t_n F_n\right)\\
s.t.\quad & F_i\geq F_{i+1},\forall i\geq0\\
& F_i>0,\forall i\geq0\\
& F_0=1.
\end{aligned}
\end{equation}
Suppose that Assumption \ref{aspt1} is satisfied and $F_{m-1}^*(m)\neq F_{m}^*(m)$. Then we have
\[
\begin{aligned}
&F_n^{*}=F_n^{*}(m),\quad n=0,1,2,...,m\\
&F_{n+1}^{*}=F_{n}^{*}\sqrt{\frac{\beta_{n+1}}{2\beta_n}},\quad n=m,m+1,...
\end{aligned}
\]
\end{thm}
\begin{proof}
Let $J^{*}=(L_{1}^{*},...,L_{k}^{*},L_{k}^{*}+1,L_{k}^{*}+2,...)$, where $(L_{1}^{*},...,L_{k}^{*})$ are optimal for Problem (\ref{optm}). The condition $F_{m-1}^*(m)\neq F_{m}^*(m)$ implies that $L_{k}^{*}=m$ and $\frac{\beta_{L_{k-1}}}{t_{L_{k-1}}}>\frac{\beta_m}{t_m}$. We shall prove that $J^{*}$ is also optimal in the infinite sense. 

Assumption \ref{aspt1} implies that, for any $i\geq m$, we have $\beta_i >\beta_{i+1}>...>\beta_{i+n}$ and $t_i<t_{i+1}<...<t_{i+n}$, where $n\in \mathbb{N}^{+}$, so it holds that
\[
\frac{\beta_i}{t_i}>\frac{\sum_{j=i}^{i+n}\beta_{j}}{\sum_{j=i}^{i+n}t_{j}}>\frac{\beta_{i+n}}{t_{i+n}}.
\]
Thus, for any $J=(L_i,i\geq 0)\in \mathcal{J}$, there exists $l\in \mathbb{N}^{+}$, such that $L_{l}\geq m$ and $\frac{\beta_{l-1}(J)}{t_{l-1}(J)}>\frac{\beta_{L_{l}}}{t_{L_{l}}}$, i.e., $J\in\mathcal{J}^{(L_{1},L_{2},...,L_{l})}$.  Let $J_{*}^{L}=(L_{1}, L_{2},...,L_{l},L_{l}+1,L_{l}+2,...)$. It follows by Proposition \ref{pro1} and Lemma \ref{lemma1} that $J_{*}^{L}\in \mathcal{J}^{(L_{1},L_{2},...,L_{l})}$ and 
\begin{equation}
\min_{\overline{F}}  g_{J}(\overline{F})\geq \min_{\overline{F}}  g_{J_{*}^L}(\overline{F}).\label{eq1}
\end{equation}
Hence, it suffices to compare $\min_{\overline{F}}  g_{J_{*}^L}(\overline{F})$ and $\min_{\overline{F}}  g_{J^{*}}(\overline{F})$.

Recall that for the optimization problem
\[
\min_{\overline{F}}  g_{J_{*}^L}(\overline{F})=\left(\sum_{n=0}^{\infty}\frac{\beta_n (J_{*}^L)}{\overline{F}_{n}}\right)\left(\sum_{n=0}^{\infty}t_n (J_{*}^L)\overline{F}_{n}\right),
\]
the solution is
\[
\overline{F}_{i}(J_{*}^L):=\sqrt{\frac{\beta_i(J_{*}^L)/t_i(J_{*}^L)}{\beta_0(J_{*}^L)/t_0(J_{*}^L)}}.
\]
Since $J_{*}^{L}=(L_{1}, L_{2},...,L_{l},L_{l}+1,L_{l}+2,...)$, we have the recursion formula
\[
\overline{F}_{n}(J_{*}^L)=\overline{F}_{l}(J_{*}^L)\sqrt{\frac{\beta_{n}t_{L_{l}}}{t_{n}\beta_{L_{l}}}}
\]
for any $n>l$. Let 
\[
F_{L_i}(J_{*}^L)=F_{L_{i}+1}(J_{*}^L)=...=F_{L_{i+1}-1}(J_{*}^L)=\overline{F}_{l}(J_{*}^L), \quad \forall i\geq 0,
\]
corresponding to $J_{*}^L$, and it follows that
\[
F_{n}(J_{*}^L)=F_{L_i}(J_{*}^L)\sqrt{\frac{\beta_{n}t_{L_{l}}}{t_{n}\beta_{L_{l}}}}
\]
for any $n>L_i$. Therefore, we have
\begin{align}
&\min_{\overline{F}}  g_{J_{*}^L}(\overline{F})\nonumber\\
&=\left(\sum_{n=0}^{L_{l}}\frac{\beta_n}{F_n(J_{*}^L)}+\sum_{n=L_{l}+1}^{\infty}\frac{\beta_n}{F_{L_{l}}(J_{*}^L)\sqrt{\frac{\beta_{n}t_{L_{l}}}{t_{n}\beta_{L_{l}}}}}\right)\nonumber\\
&\quad \times \left(\sum_{n=0}^{L_{l}}t_{n}F_n(J_{*}^L)+\sum_{n=L_{l}+1}^{\infty}t_{n}F_{L_{l}}(J_{*}^L)\sqrt{\frac{\beta_{n}t_{L_{l}}}{t_{n}\beta_{L_{l}}}}\right)\nonumber\\
&=\left(\sum_{n=0}^{L_{l}-1}\frac{\beta_n}{F_n(J_{*}^L)}+\frac{(1+c)\beta_{L_{l}}}{F_{L_{l}}(J_{*}^L)}\right)\left(\sum_{n=0}^{L_{l}-1}t_{n}F_n(J_{*}^L)+(1+c)t_{L_{l}}F_{L_{i}}(J_{*}^L)\right),\nonumber
\end{align}
where 
\[
c=\frac{\sum_{n=L_{l}+1}^{\infty}\sqrt{\beta_{n}t_{n}}}{\sqrt{\beta_{L_{l}}t_{L_{l}}}}<\infty.
\]
For $J^{*}=(L_{1}^{*},...,L_{k}^{*},L_{k}^{*}+1,L_{k}^{*}+2,...)$, we get analogously that
\begin{align*}
&\min_{\overline{F}}  g_{J^{*}}(\bar{F})\\
&=\left(\sum_{n=0}^{L_{l}-1}\frac{\beta_n}{F_n(J^{*})}+\frac{(1+c)\beta_{L_{l}}}{F_{L_{l}}(J^{*})}\right)\left(\sum_{n=0}^{L_{l}-1}t_{n}F_n(J^{*})+(1+c)t_{L_{l}}F_{L_{i}}(J^{*})\right).
\end{align*}

Since $(L_{1}^{*},...,L_{k}^{*})$ is an optimal sequence for Problem (\ref{optm}) and $F_{m-1}^*(m)\neq F_{m}^*(m)$, it holds that
\[
\frac{\beta_{k-1}(J^{*})}{t_{k-1}(J^{*})}>\frac{\beta_{m}}{t_{m}}>...>\frac{\beta_{L_{l}}}{t_{L_{l}}}=\frac{(1+c)\beta_{L_{l}}}{(1+c)t_{L_{l}}}.
\]
Then, it follows from Proposition \ref{cui} that $(L_{1}^{*},...,L_{k}^{*},L_{k}^{*}+1,...,L_{l})$ is an optimal sequence for the problem
\[
\min_{F}\left(\sum_{n=0}^{L_{l}-1}\frac{\beta_n}{F_n}+\frac{(1+c)\beta_{L_{l}}}{F_{L_{l}}}\right)\left(\sum_{n=0}^{L_{l}-1}t_{n}F_n+(1+c)t_{L_{l}}F_{L_{l}}\right). 
\]
Hence, we have
\begin{equation}
\min_{\overline{F}}  g_{J_{*}^L}(\overline{F})\geq \min_{\overline{F}}  g_{J^{*}}(\overline{F}).\label{eq2}
\end{equation}
Therefore, combining (\ref{eq1}) and (\ref{eq2}), we conclude that for any $J\in \mathcal{J}$, there is a corresponding $J_{*}^L$, such that
\[
\min_{\overline{F}}  g_{J}(\overline{F})\geq\min_{\overline{F}}  g_{J_{*}^L}(\overline{F})\geq \min_{\overline{F}}  g_{J^{*}}(\overline{F}).
\]
This indicates that $J^*=(L_{1}^{*},...,L_{k}^{*},L_{k}^{*}+1,L_{k}^{*}+2,...)$ is optimal for Problem (\ref{Problem3}) with an infinite horizon. As $(L_{1}^{*},...,L_{k}^{*})$ is optimal for the $m$-truncated problem (\ref{optm}), we obtain that $F_n^{*}=F_n^{*}(m)$ for $n=0,1,2,...,m$ and $F_{n+1}^{*}=F_{n}^{*}\sqrt{\frac{\beta_{n+1}}{2\beta_n}}$ for $n=m+1,m+2,...$, which completes the proof.
\end{proof}

Theorem \ref{thm4.1} means that we can utilize an $m$-truncated algorithm to find the optimal distribution $F_{n}^*$, $n=0,1,..,m$ and then calculate $F_{n}^*$, $n=m+1,m+2,...$ recursively using 
\begin{equation}
F_{n+1}^{*}=F_{n}^{*}\sqrt{\frac{\beta_{n+1}}{2\beta_n}}\approx 2^{\frac{-2p-1}{2}}F_{n}^{*}.\label{Fnapprox}
\end{equation}
Here, the value of $m$ can be determined adaptively starting with $m=1$, until the conditions $4^p-\epsilon \leqslant \beta_m/\beta_{m+1}\leqslant 4^p+\epsilon$ and $F_{m-1}^*(m)\neq F_{m}^*(m)$ are satisfied. In our setting, we have $\beta=O(2^{-2pn})$, $p>1/2$, so $\beta_n/\beta_{n+1}\approx 4^p$ holds for sufficiently large $n$, which implies that $4^p-\epsilon \leqslant \beta_n/\beta_{n+1}\leqslant 4^p+\epsilon$. Here, $\epsilon$ controls the accuracy of the approximation (\ref{Fnapprox}). Clearly, the smaller $\epsilon$ is, the more accurate the approximation is. However, if $\epsilon$ is too small, then a large value of $m$ is needed to satisfy the condition $4^p-\epsilon \leqslant \beta_n/\beta_{n+1}\leqslant 4^p+\epsilon$, $n\geq m$, which increases the computational cost of prior estimation. In practical applications, we find that $\epsilon=0.5$ is typically satisfactory for a time-discrete scheme with $p=1$. The other condition $F_{m-1}^*(m)\neq F_{m}^*(m)$ is also important, and it implies that
\[
\frac{\beta_{k-1}(J^*)}{t_{k-1}(J^*)}>\frac{\beta_{k}(J^*)}{t_{k}(J^*)}=\frac{\beta_m}{t_{m}}
\]
where $J^{*}=(L_{1}^*,...,L_{k}^*)$ is optimal for the $m$-truncated distribution. This condition is easy to satisfy due to the lemma below:

\begin{lemma}\label{lemma3}
If Assumption \ref{aspt1} is satisfied and $F_{m-1}^*(m)=F_{m}^*(m)$, then we have
\[
F_{m}^*(m+1)\neq F_{m+1}^*(m+1),
\]
where $F^*(m+1)$ is the optimal distribution for the $(m+1)$-truncated problem. 
\end{lemma}
\begin{proof}
Let $J^{*}=(L_{1}^*,...,L_{k}^*)$ be optimal for the $m$-truncated distribution, where $L_{k}^*<m$ due to $F_{m-1}^*(m)=F_{m}^*(m)$. Since $\beta_{k}(J^*)=\sum_{i=L_{k}^*}^{m}\beta_i>\beta_m$ and 
\[
t_{k}(J^*)=\sum_{i=L_{k}^*}^{m}t_i=\sum_{i=L_{k}^*}^{m}2^{-(m-i)}t_m<2t_m,
\]
we have
\[
\frac{\beta_{k}(J^*)}{t_{k}(J^*)}>\frac{\beta_m}{2t_m}>\frac{\beta_{m+1}}{t_{m+1}},
\]
where $\beta_{m}>\beta_{m+1}$ due to Assumption \ref{aspt1}. Thus, it follows from Proposition \ref{cui} that $(L_{1}^{*},...,L_{k}^{*},m+1)$ is optimal for the $(m+1)$-truncated distribution with $F_{m}^*(m+1)\neq F_{m+1}^*(m+1)$.
\end{proof}

\begin{remark}
Lemma \ref{lemma3} implies that the condition $F_{m-1}^*(m)\neq F_{m}^*(m)$ is much easier to satisfy than the condition $4^p-\epsilon \leqslant \beta_n/\beta_{n+1}\leqslant 4^p+\epsilon$, $n\geq m$. In most cases we encounter, when the latter is satisfied, the former is also satisfied. 
\end{remark}

Based on the above discussion, we are now ready to provide an adaptive algorithm (Algorithm \ref{algo}) to calculate the optimal distribution $F^*$ with an infinite horizon, which is an extension of the algorithm in Cui et.al.\cite{CLZZ} with an adaptive value of $m$.  In Algorithm \ref{algo}, $BetaEstn(m)$ is a self-defined function to estimate $\beta_m$. We need to estimate only $\beta_{n}$, $n=0,1,...,m,m+1$ rather than all $\beta_{n}$, saving a large amount of computational time. Here, the output $P_n$ is the optimal distribution $F^*_{n}$ for $n=0,1,2,...,m$. The condition $L_k=R_k$ with $k=Index$ is equivalent to $F_{m-1}^*(m)\neq F_{m}^*(m)$. Note that for the original algorithm in Cui et.al.\cite{CLZZ}, to calculate the $m$-truncated optimal distribution, one needs to compute the $2$-nd, $3$-rd,...,$(m-1)$-th truncated optimal distributions in order. Thus, for our adaptive algorithm, the computational cost of calculating the optimal distribution is exactly the same as that of calculating the $m$-truncated optimal distribution with the same value of $m$.

\begin{algorithm}
\caption{Adaptive algorithm for optimal distributions}
\label{algo}
\begin{algorithmic}

\State $k\leftarrow 0, Index\leftarrow -1$
\For{$m=0:10$}
   \State $L_m\leftarrow m,R_m\leftarrow m,t_m \leftarrow 2^{m}$
   \State $V_m\leftarrow(\beta_m/t_m)^{1/2}$
   \If{$m=0$}
       \State $\beta_{m}=BetaEstn(m)$, $\beta_{m+1}=BetaEstn(m+1)$
   \Else
       \State $\beta_{m+1}=BetaEstn(m+1)$
   \EndIf  
   \State $Index\leftarrow Index+1$
   \While{$k<Index$}
      \If{$V_k\leqslant V_{k+1}$}
         \State $R_k\leftarrow R_{k+1}$
         \State $V_k\leftarrow [(\sum_{n=L_k}^{R_k}\beta_n)/(\sum_{n=L_k}^{R_k}t_n)]^{1/2}$
         \State $L_{k+1},R_{k+1},V_{k+1}\leftarrow \emptyset$
         
         \If{$k\neq0$}
            \State $k\leftarrow k-1$
         \EndIf
      \Else
         
         \State $k\leftarrow k+1$
      \EndIf
   \EndWhile
   \If{$m>0$ and $\mid \beta_m/\beta_{m+1}-4^p\mid<\epsilon$ and $L_k=R_k$}
      \State break
   \EndIf
\EndFor    
\State $F_n\leftarrow V_k,n=L_k,...,R_k,k=0,1,...,Index$
\State $P_n=F_n/F_0$, $n=0,1,...,m$\\
\Return $m$, $P_n$, $n=0,1,...,m$

\end{algorithmic}
\end{algorithm}

\section{Numerical experiment}

In this section, we evaluate the efficiency of our adaptive algorithm. The algorithm is combined with several time-discrete schemes applied to SDE models, including the Black-Scholes model, the Heston model and the Heston-Hull-White model. These models are widely used in finance. 

We consider two unbiased estimators to approximate the expectation $E(Y)$: the coupled sum estimator
\[
Z=\sum_{n=0}^{N}\frac{\Delta_n}{P(N\geq n)},
\]
where $\Delta_n=Y_n-Y_{n-1}$, and the independent sum estimator
\[
\widetilde{Z}=\sum_{n=0}^{N}\frac{\widetilde{\Delta}_n}{P(N\geq n)},
\]
where $\widetilde{\Delta}_n=\widetilde{Y}_n-\widetilde{Y}_{n-1}$. As discussed, the efficiency of the above two estimators is measured by the product $E(\tau) \times var(Z)$, where $E(\tau)$ is the average computational time and $var(Z)$ is the variance of the unbiased estimator $Z$. Furthermore, this product relies on the probabilities $F_n=P(N\geq n)$. We compare three choices of $F_n$, denoted by $Dist1$, $Dist2$ and $Dist3$, respectively. The first choice is
\[
Dist1:=(2^{-n(2p+1)/2}: n\geq 0),
\]
which represents the subcanonical distribution of $N$ introduced in Section 4 of Rhee and Glynn \cite{RG}. The second choice is
\[
Dist2:=(F_n^{*}(m): m=7, n\geq 0),
\]
where $F_n^{*}(m)$ is the $m$-truncated optimal distribution of $N$ with $m=7$, calculated through the original algorithm in Cui et.al \cite{CLZZ} and the tail distribution follows $F_{n+1}^{*}(m)=F_{n}^{*}(m)\times 2^{-(2p+1)/2}$, $n=m,m+1,..$. The third choice is
\[
Dist3:=(F_0^{*},F_1^{*},...,F_m^{*},F_m^{*}\times 2^{-(2p+1)/2},...),
\]
where $F_n^{*}$ is the optimal distribution with an infinite horizon, calculated through our algorithm with an adaptive $m$, and we set $\epsilon=0.5$.

Note that the calculation of $Dist2$ and $Dist3$ for the unbiased estimators $Z$ and $\widetilde{Z}$ requires prior estimation of the first $m$ terms of 
\[
v_n=E[(Y_{n-1}-Y)^2]-E[(Y_n-Y)^2]
\]
and 
\[
\widetilde{v}_n=var(Y_{n}-Y_{n-1})+(E(Y)-E(Y_{n-1}))^2-(E(Y)-E(Y_{n}))^2,
\] 
respectively, where $Y_n$ is the approximation of $Y$ based on a time-discrete scheme with step size $T/n$. Since it is generally difficult to generate $Y$ exactly, we approximate $Y\approx Y_{10}$. For estimation of the quantities $v_{n}$, it typically suffices to obtain a highly accurate estimation based on $5\times 10^5$ samples of $(Y_{n-1}, Y_n, Y)$; for estimation of the quantities $\widetilde{v}_{n}$, we generate at least $10^6$ samples of $(Y_{n-1},Y_n, Y)$. In our numerical test, we let $Y=e^{-r}(S(1)-1,0)^+$, which corresponds to the payoff of the standard European call option. Here, $S$ is the path of the SDE model and $r$ is the interest rate.

Our aim is to investigate which choice of distribution $F_n$ is most efficient for the unbiased estimators  and has the least computational time in prior estimation. All numerical experiments are performed in MATLAB.

\subsection{Black-Scholes model}

The Black-Scholes model (Black and Scholes \cite{BS}) is given by
\[
dS(t)=r S(t)dt+\sigma S(t)dB(t),
\]
where $r,\sigma\in \mathbb{R}^{+}$ and $B=(B(t):t\geq0)$ is the standard Brownian motion. We consider the Milstein scheme to approximate the path:
\[
\begin{aligned}
S_h((j+1)h)=&S_h(jh)+r S_h(jh)h+\sigma S_h(jh)(B((j+1)h)-B(jh))\\
&+\frac{1}{2}\sigma^2 S_h(jh)((B((j+1)h)-B(jh))^2-h),
\end{aligned}
\]
where $j=1,2,...,T/h-1$, and $S_{h}(0)=S(0)$. 

Table \ref{table:1} compares the coupled sum estimator $Z$ from three different distributions $Dist1$, $Dist2$ and $Dist3$ of $N$. Here, $Var$ represents the variance of $\frac{1}{n}\sum_{i=1}^{n}Z_{i}$ with sample size $n=10^6$ and $time$ is the corresponding real computational time in seconds required to compute these quantities.  From Table \ref{table:1}, we observe that $Dist3$ from our adaptive algorithm leads to a smaller value of $Var\times time$ in comparison with $Dist1$, and is thus more computationally efficient. Furthermore, both $Dist2$ and $Dist3$ have a similar value of $Var\times time$, as their distributions are very similar. However, the computational time in prior estimation for $Dist3$ is far less, because its corresponding $m$ is much smaller. Hence, we prefer $Dist3$ to the other distributions. 

Moreover, we see that $\beta_{n}=O(2^{-2pn})$ with $p=1$. The stopping conditions that $F_{m-1}^{*}(m)\neq F_{m}^{*}(m)$ and $3.5<\beta_{m}/\beta_{m+1}<4.5$ are satisfied for a minimum of $m=1$, so the adaptive value of $m$ from $Dist3$ is $1$. For $n>m$, we can check that the condition $3.5<\beta_{n}/\beta_{n+1}<4.5$ is also satisfied; thus it is reasonable to approximate $F_{n+1}^{*}=F_{n}^{*}\sqrt{\frac{\beta_{n+1}}{2\beta_n}}\approx 2^{-3/2}F_{n}^{*}$, which means the above stopping conditions with $m=1$ suffices to maintain the robustness of our adaptive algorithm. For the independent sum estimator $\widetilde{Z}$, analogous conclusions hold, as shown in Table \ref{table:2}.

\begin{table}
\begin{center}
\setlength\tabcolsep{1pt}
\caption{Coupled sum estimator $Z$; Black-Scholes model; $r=0.05$, $\sigma=0.20$, $T=1$, $S(0)=1$; sample size $10^6$.}
\label{table:1}
\begin{tabular}{|c|c|c|c|c|c|c|c|}
\hline
$Z$&$m$&\multicolumn{2}{|c|}{$var$}&\multicolumn{2}{|c|}{$time$}&\multicolumn{2}{|c|}{$var\times time$}\\
\hline
Dist1&$\diagdown$&\multicolumn{2}{|c|}{$2.21\times10^{-8}$}&\multicolumn{2}{|c|}{$26.98$}&\multicolumn{2}{|c|}{$5.98\times10^{-7}$}\\
Dist2&$7$&\multicolumn{2}{|c|}{$2.70\times10^{-8}$}&\multicolumn{2}{|c|}{$11.91$}&\multicolumn{2}{|c|}{$3.21\times10^{-7}$}\\
Dist3&$1$&\multicolumn{2}{|c|}{$2.68\times10^{-8}$}&\multicolumn{2}{|c|}{$11.92$}&\multicolumn{2}{|c|}{$3.19\times10^{-7}$}\\
\hline n&0&1&2&3&4&5&6\\
\hline  {$\beta_n$}&0.0032&{$5.51\times10^{-5}$}&{$1.48\times10^{-5}$}&{$3.74\times10^{-6}$}&{$1.08\times10^{-6}$}&{$2.49\times10^{-7}$}&{$5.73\times10^{-8}$}\\
\hline
Dist1&1.000&0.3536&0.1250&0.0442&0.0156&0.0055&0.0020\\
Dist2&1.000&0.0293&0.0108&0.0038&0.0015&0.0005&0.0002\\
Dist3&1.000&0.0293&0.0104&0.0037&0.0013&0.0005&0.0002\\
\hline
\end{tabular}
\end{center}
\end{table}

\begin{table}
\begin{center}
\setlength\tabcolsep{1pt}
\caption{Independent sum estimator $\widetilde{Z}$; Black-Scholes model; $r=0.05$, $\sigma=0.20$, $T=1$ $S(0)=1$; sample size $10^6$.}
\label{table:2}
\begin{tabular}{|c|c|c|c|c|c|c|c|}
\hline
$\widetilde{Z}$&$m$&\multicolumn{2}{|c|}{$var$}&\multicolumn{2}{|c|}{$time$}&\multicolumn{2}{|c|}{$var\times time$}\\
\hline
Dist1&$\diagdown$&\multicolumn{2}{|c|}{$1.99\times10^{-8}$}&\multicolumn{2}{|c|}{$24.92$}&\multicolumn{2}{|c|}{$4.96\times10^{-7}$}\\
Dist2&$7$&\multicolumn{2}{|c|}{$2.40\times10^{-8}$}&\multicolumn{2}{|c|}{$2.29$}&\multicolumn{2}{|c|}{$5.51\times10^{-8}$}\\
Dist3&$1$&\multicolumn{2}{|c|}{$2.41\times10^{-8}$}&\multicolumn{2}{|c|}{$2.01$}&\multicolumn{2}{|c|}{$4.84\times10^{-8}$}\\
\hline n&0&1&2&3&4&5&6\\
\hline  {$\beta_n$}&0.00305&{$2.62\times10^{-5}$}&{$7.46\times10^{-6}$}&{$2.12\times10^{-6}$}&{$5.27\times10^{-7}$}&{$1.43\times10^{-7}$}&{$3.43\times10^{-8}$}\\
\hline
Dist1&1.000&0.3536&0.1250&0.0442&0.0156&0.0055&0.0020\\
Dist2&1.000&0.0207&0.0078&0.0029&0.0010&0.0004&0.0001\\
Dist3&1.000&0.0207&0.0073&0.0026&0.0009&0.0003&0.0001\\
\hline
\end{tabular}
\end{center}
\end{table}

The numerical examples from Tables \ref{table:1} and \ref{table:2} do not involve $F_{n}^{*}=F_{n+1}^{*}$ for some $n$, so the value of $m$ required in our adaptive algorithm is quite small. Tables \ref{table:7} and \ref{table:8} present opposite cases, where  $F_{n}^{*}=F_{n+1}^{*}$ can occur. We see that for both unbiased estimators $Z$ and $\widetilde{Z}$, the conclusions are similar to that in the previous numerical studies. Particularly, the computational efficiency of $Dist3$ measured by $var\times time$ is superior to that of $Dist1$. However, in this case, the value of $m$ needed can be large (e.g., $m=5$ in Table \ref{table:8}).

\begin{table}
\begin{center}
\caption{Coupled sum estimator $Z$; Black-Scholes model; $r=0.05$, $\sigma=2$, $T=1$, $S(0)=1$; sample size $10^6$.}
\label{table:7}
\begin{tabular}{|c|c|c|c|c|c|c|c|}
\hline
$Z$&$m$&\multicolumn{2}{|c|}{$var$}&\multicolumn{2}{|c|}{$time$}&\multicolumn{2}{|c|}{$var\times time$}\\
\hline
Dist1&$\diagdown$&\multicolumn{2}{|c|}{$6.13\times10^{-4}$}&\multicolumn{2}{|c|}{$27.28$}&\multicolumn{2}{|c|}{$0.0167$}\\
Dist2&$7$&\multicolumn{2}{|c|}{$1.25\times10^{-4}$}&\multicolumn{2}{|c|}{$88.67$}&\multicolumn{2}{|c|}{$0.0111$}\\
Dist3&$2$&\multicolumn{2}{|c|}{$1.37\times10^{-4}$}&\multicolumn{2}{|c|}{$80.27$}&\multicolumn{2}{|c|}{$0.0110$}\\
\hline n&0&1&2&3&4&5&6\\
\hline  {$\beta_n$}&12.03&{$10.25$}&{$37.99$}&{$8.97$}&{$2.55$}&{$0.71$}&{$0.20$}\\
\hline
Dist1&1.000&0.3536&0.1250&0.0442&0.0156&0.0055&0.0020\\
Dist2&1.000&0.8175&0.8175&0.3053&0.1151&0.0430&0.0162\\
Dist3&1.000&0.8175&0.8175&0.2890&0.1022&0.0361&0.0128\\
\hline
\end{tabular}
\end{center}
\end{table}

\begin{table}
\begin{center}
\caption{Independent sum estimator $\widetilde{Z}$; Black-Scholes model; $r=0.05$, $\sigma=2$, $T=1$ $S(0)=1$; sample size $10^6$.}
\label{table:8}
\begin{tabular}{|c|c|c|c|c|c|c|c|}
\hline
$\widetilde{Z}$&$m$&\multicolumn{2}{|c|}{$var$}&\multicolumn{2}{|c|}{$time$}&\multicolumn{2}{|c|}{$var\times time$}\\
\hline
Dist1&$\diagdown$&\multicolumn{2}{|c|}{$1.68\times10^{-2}$}&\multicolumn{2}{|c|}{$25.05$}&\multicolumn{2}{|c|}{$0.4218$}\\
Dist2&$7$&\multicolumn{2}{|c|}{$4.53\times10^{-4}$}&\multicolumn{2}{|c|}{$400.71$}&\multicolumn{2}{|c|}{$0.1815$}\\
Dist3&$5$&\multicolumn{2}{|c|}{$5.25\times10^{-4}$}&\multicolumn{2}{|c|}{$421.96$}&\multicolumn{2}{|c|}{$0.2215$}\\
\hline n&0&1&2&3&4&5&6\\
\hline  {$\beta_n$}&13.82&{$26.01$}&{$64.98$}&{$87.02$}&{$35.10$}&{$19.69$}&{$5.44$}\\
\hline
Dist1&1.000&0.3536&0.1250&0.0442&0.0156&0.0055&0.0020\\
Dist2&1.000&1.000&1.000&0.8523&0.3823&0.2027&0.0753\\
Dist3&1.000&1.000&1.000&0.8523&0.3823&0.2027&0.0717\\
\hline
\end{tabular}
\end{center}
\end{table}

\subsection{Heston model}

The Heston model (Heston \cite{He}) is given by
\begin{align*}
dS(t)&=r S(t)+\sqrt{V(t)}S(t)dB_1(t)\\
dV(t)&=k(\theta-V(t))dt+\sigma \sqrt{V(t)}dB_2(t)
\end{align*}
where $r, k, \theta, \sigma\in \mathbb{R}^{+}$. Here, $B_1=(B_1(t):t\geq0)$ and $B_2=(B_2(t):t\geq0)$ are independent standard Brownian motions. 

If the model parameters satisfy the Feller boundary condition $2k\theta \geq \sigma^2$, the drift-implicit Milstein scheme ensures that the approximation of $V$ is nonnegative. It can be combined with the antithetic truncated Milstein method in Giles and Szpruch \cite{GS}, leading to the time-discrete scheme as follows:
\begin{align*}
\ln(S_{h}((j+1)h))&=\ln(S_{h}(jh))+(r-\frac{1}{2}V_{h}(jh))h+\sqrt{V_{h}(jh)}\Delta B_{1,j}+\frac{\sigma}{4} \Delta B_{1,j}\Delta B_{2,j}\\
V_{h}((j+1)h)&=V_{h}(jh)+k(\theta-V_{h}((j+1)h))h+\sigma \sqrt{V_{h}(jh)}\Delta B_{2,j}+\frac{\sigma^2}{4}(\Delta B_{2,j}^2-h)
\end{align*}
where $j=0,1,..,T/h$. 

The relevant results are shown in Tables \ref{table:3} and \ref{table:4}. As we can see, for both the coupled estimator $Z$ and the independent sum estimator $\widetilde{Z}$, the values of $var\times time$ from the distributions $Dist2$ and $Dist3$ are similar since they have very similar distributions, and their $var\times time$ are smaller than that from $Dist1$. Furthermore, the computational cost in prior estimation of $Dist3$ is smaller than that of $Dist2$. Thus, we prefer $Dist3$ to $Dist1$ and $Dist2$.

\begin{table}[htbp]
\begin{center}
\setlength\tabcolsep{1pt}
\caption{Coupled sum estimator $Z$; Heston model; $r=0.05$, $\sigma=0.25$, $k=1$, $\theta=0.04$, $T=1$ , $S(0)=1,V(0)=0.04$; sample size $10^6$.}
\label{table:3}
\begin{tabular}{|c|c|c|c|c|c|c|c|}
\hline
$Z$&$m$&\multicolumn{2}{|c|}{$var$}&\multicolumn{2}{|c|}{$time$}&\multicolumn{2}{|c|}{$var\times time$}\\
\hline
Dist1&$\diagdown$&\multicolumn{2}{|c|}{$2.60\times10^{-8}$}&\multicolumn{2}{|c|}{$54.59$}&\multicolumn{2}{|c|}{$1.42\times 10^{-6}$}\\
Dist2&$7$&\multicolumn{2}{|c|}{$4.05\times10^{-8}$}&\multicolumn{2}{|c|}{$28.70$}&\multicolumn{2}{|c|}{$1.16\times 10^{-6}$}\\
Dist3&$1$&\multicolumn{2}{|c|}{$4.48\times10^{-8}$}&\multicolumn{2}{|c|}{$26.90$}&\multicolumn{2}{|c|}{$1.20\times 10^{-6}$}\\
\hline  n&0&1&2&3&4&5&6\\
\hline  {$\beta_n$}&0.0306&{$6.19\times10^{-4}$}&{$1.55\times10^{-4}$}&{$4.07\times10^{-5}$}&{$1.09\times10^{-5}$}&{$2.97\times10^{-6}$}&{$8.23\times10^{-7}$}\\
\hline  
Dist1&1.000&0.3536&0.1250&0.0442&0.0156&0.0055&0.0020\\  Dist2&1.000&0.1006&0.0356&0.0129&0.0047&0.0017&0.0006\\  Dist3&1.000&0.1006&0.0355&0.0126&0.0044&0.0016&0.0006\\
\hline
\end{tabular}
\end{center}
\end{table}

\begin{table}[htbp]
\begin{center}
\setlength\tabcolsep{1pt}
\caption{Independent sum estimator $\widetilde{Z}$; Heston model; $r=0.05$, $\sigma=0.25$, $k=1$, $\theta=0.04$, $S(0)=1,V(0)=0.04$; sample size $10^6$.}
\label{table:4}
\begin{tabular}{|c|c|c|c|c|c|c|c|}
\hline
$\widetilde{Z}$&$m$&\multicolumn{2}{|c|}{$var$}&\multicolumn{2}{|c|}{$time$}&\multicolumn{2}{|c|}{$var\times time$}\\
\hline
Dist1&$\diagdown$&\multicolumn{2}{|c|}{$2.99\times10^{-8}$}&\multicolumn{2}{|c|}{$67.63$}&\multicolumn{2}{|c|}{$2.02\times10^{-6}$}\\
Dist2&$7$&\multicolumn{2}{|c|}{$4.47\times10^{-8}$}&\multicolumn{2}{|c|}{$31.77$}&\multicolumn{2}{|c|}{$1.42\times10^{-6}$}\\
Dist3&$1$&\multicolumn{2}{|c|}{$4.75\times10^{-8}$}&\multicolumn{2}{|c|}{$30.30$}&\multicolumn{2}{|c|}{$1.44\times10^{-6}$}\\
\hline n&0&1&2&3&4&5&6\\
\hline  {$\beta_n$}&0.0367&{$3.15\times10^{-4}$}&{$8.18\times10^{-5}$}&{$2.20\times10^{-5}$}&{$6.19\times10^{-6}$}&{$1.77\times10^{-6}$}&{$5.31\times10^{-7}$}\\
\hline
Dist1&1.000&0.3536&0.1250&0.0442&0.0156&0.0055&0.0020\\
Dist2&1.000&0.0655&0.0236&0.0087&0.0032&0.0012&0.0005\\
Dist3&1.000&0.0655&0.0232&0.0082&0.0029&0.0010&0.0004\\ 
\hline
\end{tabular}
\end{center}
\end{table}

\subsection{Heston-Hull-White model}

The Heston-Hull-White model is given by
\begin{align*}
dS(t)&=r(t) S(t) dt+\sqrt{V(t)}S(t)(\rho dB_{1}(t)+\sqrt{1-\rho^2}dB_{2}(t))\\
dV(t)&=k(\theta-V(t))dt+\sigma \sqrt{V(t)}dB_{1}(t)\\
dr(t)&=\alpha(\beta-r(t))dt+\gamma dB_{3}(t)
\end{align*}
where $B_1$, $B_2$ and $B_3$ are mutually independent Brownian motions. Here, the parameters $k,\theta,\sigma, \alpha, \beta, \gamma>0$ and $\rho\in [-1,1]$. This model is an extension of the Heston model with an interest rate following the Hull-White model (Hull and White \cite{HW}); see Grzelak and Oosterlee \cite{GO} for more detail. We consider a semi-exact scheme developed in Zheng and Pan \cite{ZP}. Let $X(T)=\ln(S(T))$ and the solution can be written as
\begin{align*}
X(T)=&X(0)+\int_{0}^{T}r(t) dt+(\frac{\rho k}{\sigma}-\frac{1}{2})\int_{0}^{T}V(t) dt\\
&+\frac{\rho}{\sigma}(V(T)-V(0)-k\theta T)+\sqrt{1-\rho^2}\sqrt{\int_{0}^{T}V(t) dt}N
\end{align*}
where $N$ is a standard normal random variable that is independent of $V$ and $r$. Here, $V(t)$ follows a noncentral chi-squared distribution as the Heston model and $r(t)$ follows a normal distribution (see Glasserman \cite{G}). Then, it is convenient to approximate
\[
\int_{0}^{T}r_t dt\approx \sum_{i=0}^{T/h-1}r_{ih}h, \int_{0}^{T}V_t dt\approx \sum_{i=0}^{T/h-1}V_{ih}h
\]
where $h$ is the time step size. Moreover, the price of the standard European call option can be expressed as
\[
E\left[e^{-\int_{0}^{T}r(t)dt}\max(S(T)-1,0)\right]
\]
where the integral can be approximated analogously. For this scheme, the theoretical convergence rate $p=1$; see Zheng and Pan \cite{ZP}. 

The results of the numerical experiments are shown in Tables \ref{table:5} and \ref{table:6}. For both  estimators $Z$ and $\widetilde{Z}$, the distribution $Dist3$ is the most satisfactory. Moreover, we find that it holds that $3.5<\beta_{n}/\beta_{n+1}<4.5$ for all $n=m,m+1,...$, demonstrating again the robustness of the adaptive algorithm.

\begin{table}[htbp]
\begin{center}
\setlength\tabcolsep{1pt}
\caption{Coupled sum estimator; Heston-Hull-White; $k=3,\theta=0.04,\sigma=0.25,\alpha=1,\beta=0.06,\gamma=0.5,S(0)=1,V(0)=0.04,r(0)=0.05$; sample size $10^6$.}
\label{table:5}
\begin{tabular}{|c|c|c|c|c|c|c|c|}
\hline
$Z$&$m$&\multicolumn{2}{|c|}{$var$}&\multicolumn{2}{|c|}{$time$}&\multicolumn{2}{|c|}{$var\times time$}\\
\hline
Dist1&$\diagdown$&\multicolumn{2}{|c|}{$2.07\times10^{-7}$}&\multicolumn{2}{|c|}{$18.44$}&\multicolumn{2}{|c|}{$3.82\times 10^{-6}$}\\
Dist2&$7$&\multicolumn{2}{|c|}{$1.33\times10^{-7}$}&\multicolumn{2}{|c|}{$22.42$}&\multicolumn{2}{|c|}{$2.98\times 10^{-6}$}\\
Dist3&$2$&\multicolumn{2}{|c|}{$1.32\times10^{-7}$}&\multicolumn{2}{|c|}{$22.46$}&\multicolumn{2}{|c|}{$2.98\times 10^{-6}$}\\
\hline  n&0&1&2&3&4&5&6\\
\hline  {$\beta_n$}&{$0.0314$}&{$1.37\times10^{-2}$}&{$4.71\times10^{-3}$}&{$1.30\times10^{-4}$}&{$3.44\times10^{-5}$}&{$8.69\times10^{-5}$}&{$2.20\times10^{-5}$}\\
\hline  
Dist1&1.000&0.3536&0.1250&0.0442&0.0156&0.0055&0.0020\\  Dist2&1.000&0.4671&0.1936&0.0719&0.0262&0.0093&0.0033\\  Dist3&1.000&0.4671&0.1936&0.0685&0.0242&0.0086&0.0030\\
\hline
\end{tabular}
\end{center}
\end{table}

\begin{table}[htbp] 
\begin{center} 
\setlength\tabcolsep{1pt}
\caption{Independent sum estimator; Heston-Hull-White; $k=3,\theta=0.04,\sigma=0.25,\alpha=1,\beta=0.06,\gamma=0.5,S(0)=1,V(0)=0.04,r(0)=0.05$; sample size $10^6$.}
\label{table:6}
\begin{tabular}{|c|c|c|c|c|c|c|c|}
\hline
$Z$&$m$&\multicolumn{2}{|c|}{$var$}&\multicolumn{2}{|c|}{$time$}&\multicolumn{2}{|c|}{$var\times time$}\\
\hline
Dist1&$\diagdown$&\multicolumn{2}{|c|}{$1.56\times10^{-7}$}&\multicolumn{2}{|c|}{$21.50$}&\multicolumn{2}{|c|}{$3.35\times 10^{-6}$}\\
Dist2&$7$&\multicolumn{2}{|c|}{$1.17\times10^{-7}$}&\multicolumn{2}{|c|}{$26.23$}&\multicolumn{2}{|c|}{$3.07\times 10^{-6}$}\\
Dist3&$2$&\multicolumn{2}{|c|}{$1.19\times10^{-7}$}&\multicolumn{2}{|c|}{$27.12$}&\multicolumn{2}{|c|}{$3.23\times 10^{-6}$}\\
\hline  n&0&1&2&3&4&5&6\\
\hline  {$\beta_n$}&{$0.0322$}&{$1.23\times10^{-2}$}&{$3.67\times10^{-3}$}&{$9.85\times10^{-4}$}&{$2.53\times10^{-4}$}&{$6.40\times10^{-5}$}&{$1.61\times10^{-5}$}\\
\hline  
Dist1&1.000&0.3536&0.1250&0.0442&0.0156&0.0055&0.0020\\
Dist2&1.000&0.4368&0.1690&0.0619&0.0222&0.0079&0.0028\\
Dist3&1.000&0.4368&0.1690&0.0597&0.0211&0.0075&0.0026\\
\hline
\end{tabular}
\end{center}
\end{table}

\section{Conclusion}
This paper addresses an optimization problem concerning the unbiased estimators introduced by Rhee and Glynn \cite{RG} for SDE models. Specifically, the computational efficiency of these unbiased estimators relies on the choice of the cumulative distribution function of a random variable, and the optimal distribution should minimise the product of the variance and computational time of an unbiased estimator subject to certain constraints. Based on the results in Rhee and Glynn \cite{RG} and Cui et.al.\cite{CLZZ}, we prove that under a mild assumption on the convergence of $\beta_{n}$, there is a simple representation for the optimal distribution with an infinite horizon. This result establishes a link between the $m$-truncated optimal distribution and the optimal distribution with an infinite horizon, which enables us to construct an adaptive algorithm for the optimal distribution with an adaptive value of $m$. Compared with the $m$-truncated algorithm in Rhee and Glynn \cite{RG} and Cui et.al.\cite{CLZZ}, the merits of our adaptive algorithm are as follows: first, it is capable of handling optimal distributions with an infinite horizon; second, a small value of $m$ typically suffices to produce an accurate estimate of the optimal distribution, which saves a large amount of computational time in the prior estimation of $\beta_{n}$. The efficiency of our adaptive algorithm is illustrated by several numerical examples with some well-known SDEs models in finance.

\section*{Acknowledgement}
This research is supported by National Natural Science Foundation of China (Nos. 11801504, 11801502).

\end{document}